\documentclass[a4paper,11pt]{article}
\usepackage{amssymb}
\usepackage{amsmath}
\usepackage{amsthm}
\usepackage[all,cmtip]{xy}
\allowdisplaybreaks
\usepackage{amscd}
\usepackage{graphicx}

\usepackage{ifpdf}
\ifpdf
  \usepackage[colorlinks=true,linkcolor=blue,final,backref=page,hyperindex,citecolor=red]{hyperref}
\else
  \usepackage[colorlinks,final,backref=page,hyperindex,hypertex]{hyperref}
\fi

\usepackage{geometry}
\usepackage{pxfonts}

\usepackage[shortlabels]{enumitem}
\usepackage{ifpdf}

\def\vbar{\mathchoice{\vrule height6.3ptdepth-.5ptwidth.8pt\kern- .8pt}
{\vrule height6.3ptdepth-.5ptwidth.8pt\kern-.8pt} {\vrule
height4.1ptdepth-.35ptwidth.6pt\kern-.6pt} {\vrule
height3.1ptdepth-.25ptwidth.5pt\kern-.5pt}}

\usepackage{tikz}
\usepackage[active]{srcltx}

\makeatletter

\newtheorem{thm}{Theorem}[section]
\newtheorem{lem}[thm]{Lemma}
\newtheorem{cor}[thm]{Corollary}
\newtheorem{pro}[thm]{Proposition}
\newtheorem{ex}[thm]{Example}
\newtheorem{rmk}[thm]{Remark}
\newtheorem{defi}[thm]{Definition}

\newcommand{\be }{\begin{equation}}
\newcommand{\ee }{\end{equation}}

\newcommand{\g}{\mathfrak g}

\newcommand{\huaL}{\mathcal{L}}

\newcommand{\huaO}{{\mathcal{O}}}

\newcommand{\Id}{\rm{Id}}

\newcommand{\br}[1]{   [ \cdot,    \cdot  ]   }

\newcommand{\Hom}{\mathrm{Hom}}

\newcommand{\gl}{\mathfrak {gl}}

\def\a{\alpha}

\newcommand{\ad}{\mathrm{ad}}

\newcommand{\K}{\mathbb{K}}

\begin{document}

\title{On Hom-$F$-manifold algebras and quantization}

\author{\normalsize \bf  A. Ben Hassine\small{$^{1,2}$} \footnote { Corresponding author,  E-mail: benhassine.abdelkader@yahoo.fr}, T. Chtioui\small{$^{3}$} \footnote { Corresponding author,  E-mail: chtioui.taoufik@yahoo.fr},  M.A.Maalaoui\small{$^{4}$} \footnote { Corresponding author,  E-mail: maalaouimed@gmail.com},  S. Mabrouk\small{$^{5}$} \footnote { Corresponding author,  E-mail: mabrouksami00@yahoo.fr}}
\date{{\small{$^{1}$ Department of Mathematics, Faculty of Sciences and Arts at
Belqarn, P. O. Box 60, Sabt Al-Alaya 61985, University of Bisha, Saudi Arabia \\  \small{$^{2,3,4}$    Faculty of Sciences, University of Sfax,   BP
1171, 3000 Sfax, Tunisia \\  \small{$^{5}$} Faculty of Sciences, University of Gafsa,   BP
2100, Gafsa, Tunisia
 }}}}

\maketitle

\begin{abstract}
The notion of a $F$-manifold algebras is an algebraic description of
a $F$-manifold. In this paper, we introduce the notion of Hom-$F$-manifold algebras which is generalisation of $F$-manifold algebras and Hom-Poisson algebras. We develop the representation theory of Hom-$F$-manifold algebras and generalize the notion of Hom-pre-Poisson algebras by introducing the Hom-pre-$F$-manifold algebras which give rise to a Hom-$F$-manifold algebra through the sub-adjacent commutative Hom-associative algebra and the sub-adjacent Hom-Lie algebra. Using $\huaO$-operators  on a Hom-$F$-manifold algebras we  construct a Hom-pre-$F$-manifold algebras on a module. Then,  we study Hom-pre-Lie formal deformations  of  commutative Hom-associative algebra and prove that Hom-$F$-manifold algebras  are the corresponding semi-classical limits.  Finally,  we study
Hom-Lie infinitesimal deformations and extension of Hom-pre-Lie $n$-deformation to Hom-pre-Lie $(n+1)$-deformation of a commutative Hom-associative algebra via cohomology theory.
\end{abstract}



\maketitle

\tableofcontents

\setcounter{section}{0}

\allowdisplaybreaks

\section{Introduction}\label{sec:intr}
Algebras of Hom-type appeared in the Physics literature of the 1990's, in the context of quantum
deformations of some algebras of vector fields, such as the Witt and Virasoro algebras, in connection with oscillator algebras (\cite{Aizawa&Sato,Hu}). A quantum deformation consists of replacing the usual
derivation by a $\sigma$-derivation. It turns out that the algebras obtained in this way do not satisfy
the Jacobi identity anymore, but instead they satisfy a modified version involving a homomorphism. This kind of algebras were called Hom-Lie algebras and studied by Hartwig, Larsson and
Silvestrov in \cite{Hartwig&Larsson&Silvestrov,Larsson&Silvestrov}.
 The corresponding associative type objects, called Hom-associative algebras were introduced
by Makhlouf and Silvestrov in \cite{Makhlouf&Silvestrov}. Hom-alternative and Hom-flexible algebras were
introduced  in \cite{Makhlouf&Silvestrov}. Hom-Jordan, Hom-Malcev and Hom-Poisson algebras were studied   in \cite{Yau2010,Yau2012}.
Hom-Lie algebras are widely studied in the following aspects: representation and cohomology theory \cite{Ammar&Mabrouk&Makhlouf,Benayadi&Makhlouf,Sheng,Yau2009}, deformation theory \cite{Makhlouf&Silvestrov1} and categorification theory \cite{Sheng&Chen}.

Rota–Baxter (associative) algebras, originated from the work of G. Baxter \cite{Baxter} in probability and populated
by the work of Cartier and Rota \cite{Cartier,Rota1969,Rota1995}, have also been studied in connection with many areas of mathematics and physics,
including combinatorics, number theory, operators and quantum field theory \cite{Aguiar,Bai&Guo&Ni,Guo2009,Guo2012,Guo&Keigher}. In particular Rota–Baxter algebras
have played an important role in the Hopf algebra approach of renormalization of perturbative quantum field theory of
Connes and Kreimer \cite{Connes&Kreimer,Fard&Guo&Kreimer,Fard&Guo&Manchon}, as well as in the application of the renormalization method in solving divergent problems
in number theory \cite{Guo&Zhang,Manchon&Paycha}. Furthermore, Rota–Baxter operators on a Lie algebra are an operator form of the classical
Yang–Baxter equations and contribute to the study of integrable systems \cite{Bai&Guo&Ni,Bai&Guo&Ni2010}.

The notion of Frobenius manifolds was invented by Dubrov in \cite{DS11} in order to give a geometrical expression of the WDV equations. In 1999, Hertling and Manin \cite{HerMa}
introduced the concept of $F$-manifolds as a relaxation of the conditions of Frobenius manifolds. $F$-manifolds   appear in many fields of mathematics such as singularity theory \cite{Her02}, quantum K-theory \cite{LYP}, integrable systems \cite{DS04,DS11,LPR11},  operad \cite{Merku} and so on.
Inspired by the investigation of algebraic structures of $F$-manifolds, the notion of an $F$-manifold
algebra is given by Dotsenko \cite{Dot} in 2019 to relate the operad $F$-manifold algebras to the operad
pre-Lie algebras.
Recently,  the concept of pre-$F$-manifold algebras, which
gives rise to $F$-manifold algebras, and the notion  of representations of $F$-manifold algebras, are introduced in \cite{Liu&Sheng&Bai}. The notion of $F$-manifold color algebras and theirs proprieties  is given in (\cite{Ming&Chen&Li}).

The paper is organized as follows. In Section \ref{sec:Preliminaries},  we give a backround of Hom-associative algebras, Hom-Lie algebras, Hom-pre-Lie algebras, Hom-Lie-admissible algebras, $F$-manifolds, $F$-manifolds algebras. In Section \ref{sec:F-algebras and deformations},  we define Hom-$F$-manifold-algebras and  w introduce the notion of Hom-$F$-manifold-algebras, which give rise to Hom-$F$-manifold algebras.   In Section \ref{sec:pre-F-algebras},  first we develop  (dual) representations of a Hom-$F$-manifold algebras. Then we introduce the notions of Hom-pre-$F$-manifold algebras and $\huaO$-operators on a Hom-$F$-manifold algebra. We show that on one hand, an $\huaO$-operator on a Hom-$F$-manifold algebra gives a Hom-pre-$F$-manifold algebra. On the other hand, a Hom-pre-$F$-manifold algebra naturally gives an $\huaO$-operator on the sub-adjacent Hom-$F$-manifold algebra.  Section \ref{Hom-Pre-Lie  formal deformation  of commutative Hom-pre-Lie algebras}  is devoted to the study of the   Hom-pre-Lie formal deformations of commutative Hom-associative algebras and prove that Hom-$F$-manifold algebras are the corresponding semi-classical limits. Furthermore, we study  extensions of Hom-pre-Lie $n$-deformations to Hom-pre-Lie $(n+1)$-deformations of a commutative Hom-associative algebra.\\
 In this paper, all the vector spaces are over an algebraically closed field $\mathbb K$ of characteristic $0$.
\section{Preliminaries}\label{sec:Preliminaries}
In this section, we  recall Hom-associative algebras, Hom-Lie algebras, Hom-pre-Lie algebras, Hom-Lie-admissible algebras, $F$-manifolds, $F$-manifolds algebras  (for further details we can refer to \cite{Cai&Sheng,HerMa,Liu&Song&Tang,Makhlouf&Silvestrov,ms2,Yau2010,Yau2012}).

   A  Hom-associative  algebra is a triple $(A,\cdot,\a)$, where $A$ is a vector space,   $\cdot:A\otimes A\longrightarrow A$ and $\a:A\to A$ are  a two linear maps satisfying that for all $x,y,z\in A$, the Hom-associator
$as_\a(x,y,z)=(x\cdot y)\cdot \a(z)-\a(x)\cdot(y\cdot z)=0$,
i.e.
$$\a(x)\cdot (y\cdot  z)=(x\cdot  y)\cdot \a(z).$$

Furthermore, if  $x\cdot y=y\cdot x$ for all $x,y\in A$, then $(A,\cdot,\a)$ is called a commutative Hom-associative algebra.

   A representation of commutative Hom-associative  algebra $(A,\cdot,\alpha)$ on
  a vector space $V$ with respect to $\phi\in\mathfrak{gl}(V)$ is a linear map
  $\mu :A\longrightarrow \mathfrak{gl}(V)$, such that for any
  $x,y\in A$, the following equalities are satisfied:
  \begin{eqnarray}
 \label{representation-ass1} \mu (\alpha(x))  \phi&=&\phi \mu (x),\\\label{representation-ass2}
    \mu(x\cdot y)
   \phi&=&\mu (\alpha(x)) \mu(y).
  \end{eqnarray}

 Let $(V;\mu,\phi)$ be a representation of a commutative Hom-associative algebra $(A,\cdot ,\a)$.
 In the sequel, we always assume that $\phi$ is invertible.
   Define $\mu^*:A\longrightarrow \gl(V^*)$ and $\phi^*:V^*\longrightarrow V^*$ by
$$
 \langle \mu^*(x)\xi,v\rangle=-\langle \xi,\mu(x)v\rangle,\quad \forall ~ x\in A,\xi\in V^*,v\in V.
$$

However, in general $\mu^*$ is not a representation of $A$ anymore. Define $\mu^\star:A\longrightarrow\gl(V^*)$ by
\begin{equation}\label{eq:new1}
 \mu^\star(x)(\xi):=\mu^*(\alpha(x))\big{(}(\phi^{-2})^*(\xi)\big{)},\quad\forall x\in A,\xi\in V^*.
\end{equation}

More precisely, we have
\begin{eqnarray}\label{eq:new1gen}
\langle\mu^\star(x)(\xi),u\rangle=-\langle\xi,\mu(\alpha^{-1}(x))(\phi^{-2}(u))\rangle,\quad\forall x\in A, u\in V, \xi\in V^*.
\end{eqnarray}

\begin{lem}
Under the above notation,   $(V^*,\mu^\star,(\phi^{-1})^*)$ is a representation
 of $(A,\cdot,\alpha)$ which is called the  dual representation of  $(V;\mu,\phi)$.
\end{lem}

  A    Hom-Zinbiel algebra  is a triple $(A,\diamond,\alpha)$, where
$A$ is a vector space,   $\diamond:A\otimes A\longrightarrow A$ is
a binary  multiplication   and $\alpha:A\to A$  be a linear map such that for all $x,y,z\in A$,
\begin{equation}
 \alpha(x)\diamond(y\diamond z)=(y\diamond x)\diamond \alpha(z)+(x\diamond y)\diamond \alpha(z).
\end{equation}

 \begin{lem}\label{lem:den-ass}
Let $(A,\diamond,\alpha)$ be a Hom-Zinbiel algebra. Then $(A,\cdot,\alpha)$ is a commutative Hom-associative algebra, where $x\cdot y=x\diamond y+y\diamond x$. Moreover, for $x\in A$, define $L_\diamond(x):A\longrightarrow\gl(A)$ by
\begin{equation}\label{eq:dendriform-rep}
L_\diamond( x)(y)=x\diamond y,\quad\forall~y\in A.
\end{equation}
Then $(A;L_\diamond,\alpha)$ is a representation of the commutative Hom-associative algebra $(A,\cdot,\alpha)$.
\end{lem}
   A Hom-Lie  algebra is a triple $( A ,[\cdot,\cdot],\alpha)$ consisting of vector space $ A $, a bilinear mapping $[\cdot,\cdot]: A \times A \rightarrow A $ and  a linear map $\alpha : A \rightarrow A $ such that for  all  $x,y,z\in  A $ we have
\begin{eqnarray}
&& [x,y]=-[y,x],\label{Skewsymetric}\\
&&   \circlearrowleft_{x,y,z}[\alpha(x),[y,z]]= 0~~(\textrm{Hom-Jacobi\ identity}).\label{JacobiIdentity}
\end{eqnarray}
where
$\circlearrowleft_{x,y,z}$ denotes summation over the cyclic permutation on $x,y,z$.

   A representation of Hom-Lie algebra $(A,[\cdot,\cdot],\alpha)$ on
  a vector space $V$ with respect to $\phi\in\mathfrak{gl}(V)$ is a linear map
  $\rho :A\longrightarrow \mathfrak{gl}(V)$, such that for any
  $x,y\in A$, the following equalities are satisfied:
  \begin{eqnarray}
 \label{representation1} \rho (\alpha(x))  \phi&=&\phi  \rho (x),\\\label{representation2}
    \rho([x,y])
   \phi&=&\rho (\alpha(x)) \rho(y)-\rho(\alpha(y)) \rho (x).
  \end{eqnarray}

Let $(V;\rho,\phi)$ be a representation of a Hom-Lie algebra $(A,[\cdot,\cdot],\alpha)$. Assume that $\phi$ is invertible. For all $x\in A,u\in V,\xi\in V^*$, define $\rho^*:A\longrightarrow\gl(V^*)$ as usual by
$$\langle \rho^*(x)(\xi),u\rangle=-\langle\xi,\rho(x)(u)\rangle.$$
Then, define $\rho^\star:A\longrightarrow\gl(V^*)$
\begin{eqnarray}
  \label{eq:1.3}\rho^\star(x)(\xi):=\rho^*(\alpha(x))\big{(}(\phi^{-2})^*(\xi)\big{)}.
\end{eqnarray}

\begin{lem}\cite{Cai&Sheng} 
Let $(V;\rho,\phi)$ be a representation of a Hom-Lie algebra $(A,[\cdot,\cdot],\alpha)$. Then $(V^*;\rho^\star,(\phi^{-1})^*)$ is a representation of $(A,[\cdot,\cdot],\alpha)$, which is called the  dual representation of $(V;\rho,\phi)$.
\end{lem}

   A  Hom-pre-Lie algebra is a triple $(A,\ast,\a)$, where $A$ is a vector space,    $\ast:A\otimes A\longrightarrow A$ and $\a:A\to A$ are  a two linear maps satisfying   for all $x,y,z\in A$
$$as_\a(x,y,z)=as_\a(y,x,z),$$
or equivalently
$$(x\ast  y)\ast \a(z)-\a(x)\ast (y\ast  z)=(y\ast  x)\ast  \a(z)-\a(y)\ast (x\ast  z).$$

It is obvious that any Hom-associative algebra is a  Hom-pre-Lie algebra. In addition,
a commutative Hom-pre-Lie algebra is Hom-associative.

\begin{lem}
  \label{lem:pre-Lie-Lie} Let $(A,\ast,\a)$ be a Hom-pre-Lie algebra. Define the commutator
$ [x,y]=x\ast y-y\ast x$, then $(A,[\cdot,\cdot],\a)$ is a Hom-Lie algebra,
which is called the  sub-adjacent Hom-Lie algebra of $(A,\ast,\a)$ and denoted by $A^c$. Furthermore, $L_\ast:A\rightarrow
\gl(A)$ defined by
\begin{equation}\label{eq:defiLpreLie}
L_\ast(x)y=x\ast y,\quad \forall x,y\in A
\end{equation}
 gives a representation of $A^c$ on $A$.
\end{lem}

A  Hom-Lie admissible algebra is a nonassociative Hom-algebra $(A,\ast ,\a)$ whose commutator algebra is a Hom-Lie algebra.
More precisely, it is equivalent to the following condition:
\begin{equation}
  \circlearrowleft_{x,y,z}as_\a(x,y,z)-as_\a(y,x,z)=0,\quad \forall x,y,z\in A.
\end{equation}

Obviously, a Hom-pre-Lie algebra is a Hom-Lie-admissible algebra.

Let $(A,\ast ,\a)$ be a Hom-pre-Lie algebra and $V$  a vector
space. A  representation of $A$ on $V$ with respect to $\phi\in \gl(V)$  consists of a pair
$(\rho,\mu)$, where $\rho:A\longrightarrow \gl(V)$ is a representation
of the Hom-Lie algebra $A^c$ on $V $ with respect $\phi$ and $\mu:A\longrightarrow \gl(V)$ is a linear
map satisfying \begin{align}\label{multiplicativityRep}&\phi\mu(x)=\mu(\a(x))\phi,\\\label{representation condition 2}
 &\rho(\a(x))\mu(y)-\mu(\a(y))\rho(x)=\mu(x\ast  y)\phi-\mu(\alpha(y))\mu(x), \quad \forall~x,y\in A.
\end{align}

We denote such  a representation by $(V;\rho,\mu,\phi)$. Let $R_\ast:A\rightarrow
\gl(A)$,
$R_\ast(x)y=y\ast  x,$ for all $x, y\in A$. Then
$(A;L_\ast,R_\ast,\a)$ is a representation, which we call the
adjoint representation.

Let $(V;\rho,\mu,\phi)$ be a representation of a Hom-pre-Lie algebra $(A,\ast,\alpha)$. Assume that $\phi$ is invertible. For all $x\in A,u\in V,\xi\in V^*$, define $\rho^*:A\longrightarrow\gl(V^*)$ and $\mu^*:A\longrightarrow\gl(V^*)$ as usual by
$$\langle \rho^*(x)(\xi),u\rangle=-\langle\xi,\rho(x)(u)\rangle,\quad \langle \mu^*(x)(\xi),u\rangle=-\langle\xi,\mu(x)(u)\rangle.$$
Then define $\rho^\star:A\longrightarrow\gl(V^*)$ and $\mu^\star:A\longrightarrow\gl(V^*)$ by
\begin{eqnarray}
  \label{eq:1.3}\rho^\star(x)(\xi):=\rho^*(\alpha(x))\big{(}(\phi^{-2})^*(\xi)\big{)},\\
   \label{eq:1.4}\mu^\star(x)(\xi):=\mu^*(\alpha(x))\big{(}(\phi^{-2})^*(\xi)\big{)}.
\end{eqnarray}

\begin{lem}\cite{Liu&Song&Tang}\label{dual-rep}
Under the above notations, $(V^*;\rho^\star-\mu^\star,-\mu^\star,(\phi^{-1})^*)$ is a representation of $(A,\ast,\alpha)$, which is called the    dual representation  of $(V;\rho,\mu,\phi)$.
\end{lem}

Let $(V,\rho,\mu,\phi)$ be a representation of a Hom-pre-Lie algebra $(A,\ast,\alpha)$. The set of $(n+1)$-cochains is given by
\begin{equation}
 C^{n+1}(A;V)=\Hom(\wedge^n A\otimes A,V),\quad
 \forall n\geq 0.
\end{equation}
For all $f \in C^n(A;V),~ x_1,\dots,x_{n+1} \in A$, define the operator
$\partial:C^n(A;V)\longrightarrow C^{n+1}(A;V)$ by
\begin{eqnarray}\label{eq:12.1}
 &&(\partial f)(x_1,\dots,x_{n+1})\\
 \nonumber&=&\sum_{i=1}^n(-1)^{i+1}\rho(x_i)f(\alpha^{-1}(x_1),\dots,\widehat{\alpha^{-1}(x_i)},\dots,\alpha^{-1}(x_{n+1}))\\
\nonumber &&+\sum_{i=1}^n(-1)^{i+1}\mu(x_{n+1})f(\alpha^{-1}(x_1),\dots,\widehat{\alpha^{-1}(x_i)},\dots,\alpha^{-1}(x_n),\alpha^{-1}(x_i))\\
\nonumber &&-\sum_{i=1}^n(-1)^{i+1}\phi f(\alpha^{-1}(x_1),\dots,\widehat{\alpha^{-1}(x_i)}\dots,\alpha^{-1}(x_n),\alpha^{-2}(x_i)\ast \alpha^{-2}(x_{n+1}))\\
\nonumber&&+\sum_{1\leq i<j\leq n}(-1)^{i+j}\phi f([\alpha^{-2}(x_i),\alpha^{-2}(x_j)],\alpha^{-1}(x_1),\dots,\widehat{\alpha^{-1}(x_i)},\dots,\widehat{\alpha^{-1}(x_j)},\dots,\alpha^{-1}(x_{n+1})).
\end{eqnarray}

\begin{thm}\label{thm:operator}
The operator $\partial:C^n(A;V)\longrightarrow C^{n+1}(A;V)$ defined as above satisfies $\partial\circ\partial=0$.
\end{thm}

Denote the set of closed $n$-cochains by $Z^n(A;V)$ and the set of exact $n$-cochains by $B^n(A;V)$. We denote by $H^n(A;V)=Z^n(A;V)/B^n(A;V)$ the corresponding cohomology groups of the Hom-pre-Lie algebra $(A,\ast,\alpha)$ with the coefficient in the representation $(V,\rho,\mu,\phi)$.

In \cite{ms2}, the authors defined a  Hom-Poisson algebra as a tuple $(A,[\cdot,\cdot],\mu,\alpha)$ consists of
a Hom-Lie algebra $(A,[\cdot,\cdot],\alpha)$ and
a commutative Hom-associative algebra $(A,\mu,\alpha)$
obeying to  the   Hom-Leibniz identity
\begin{equation}
\label{homleibniz}
[\alpha(x),y\cdot z]=[x,y]\cdot \alpha(z)+\alpha(y)\cdot [x,z].
\end{equation}

The notion of $F$-manifold is introduced by  C. Hertling and Y. I. Manin  (\cite{HerMa}) were introduced a weak version of Frobenius  manifolds.

 An $F$-manifold is a pair $(M,\circ)$, where $M$ is a manifold, $\circ$ is a smooth bilinear commutative,
associative multiplication on the tangent sheaf $TM$, such that the Hertling-Manin relation holds
\begin{equation*}
P_{X_1\circ  X_2}(X_3,X_4)=X_1\circ P_{X_2}(X_3,X_4) +  X_2\circ P_{X_1}(X_3,X_4) ,
\end{equation*}

where $P_{X_1}(X_2,X_3)=[X_1,X_2\circ X_3]-[X_1,X_2]\circ X_3- X_2\circ[X_1,X_3]$ measures to what extent the product~$\circ$ and the usual Lie bracket of vector fields fail the Poisson algebra axioms.

In \cite{Dot}, the author gives an algebraic description of $F$-manifold constructing in this way   $F$-manifold algebras. By   a definition, an  { $F$-manifold algebra} is a triple $(A,\cdot,[\cdot,\cdot])$, where $(A,\cdot)$ is a commutative associative algebra and $(A,[\cdot,\cdot])$ is a Lie algebra, such that for all $x,y,z,w\in A$, the Hertling-Manin relation holds
:
\begin{equation}\label{eq:HM}
\mathcal L(x\cdot y, z,w)=x\cdot \mathcal L(y,z,w)+y\cdot \mathcal L(x,z,w),
\end{equation}
where $ \mathcal L(x,y,z)$ is the  Leibnizator define by
\begin{equation}
 \mathcal L(x,y,z)=[x,y\cdot z]-[x,y]\cdot z-y\cdot [x,z].
\end{equation}

\section{Hom-$F$-manifold algebras: Definitions and constructions}\label{sec:F-algebras and deformations}


In this section, we introduce the notion of Hom-$F$-manifold algebras and  various examples are given.
\begin{defi}
A {  Hom-$F$-manifold algebra} is a tuple $(A,\cdot,[\cdot,\cdot],\a)$, where $(A,\cdot,\a)$ is a commutative Hom-associative algebra and $(A,[\cdot,\cdot],\a)$ is a Hom-Lie algebra, such that for all $x,y,z,w\in A$:
\begin{equation}\label{eq:HM relation}
\mathcal L(x\cdot y,\a(z),\a(w))=\a^2(x)\cdot \mathcal L(y,z,w)+\a^2(y)\cdot \mathcal L(x,z,w),
\end{equation}
where $ \mathcal L(x,y,z)$ is the Hom-Leibnizator define by
\begin{equation}
 \mathcal L(x,y,z)=[\a(x),y\cdot z]-[x,y]\cdot \a(z)-\a(y)\cdot [x,z].
\end{equation}
The identity  \eqref{eq:HM relation} is called the Hom-Hertling-Manin relation.
\end{defi}
A  Hom-$F$-manifold algebra is called a regular  Hom-$F$-manifold algebra if $\alpha$ is an algebra automorphism.\\
Let  $x,y\in A$ such that $\a(x)=x, \a(y)=y$ then the map $ \mathcal L(\cdot,x,y): A\to A$   is $\a^2$-derivation in the Hom-associative algebra $(A,\cdot,\a)$.

\begin{defi}
Let $(A,\cdot_A,[\cdot,\cdot]_A,\a_A)$ and $(B,\cdot_B,[\cdot,\cdot]_B,\a_B)$ be two Hom-$F$-manifold algebras. A {  homomorphism} between $A$ and $B$ is a linear map $\varphi:A\rightarrow B$ such that
\begin{eqnarray}\varphi\a_A&=&\a_B\varphi,\\
 \varphi(x\cdot_A y)&=&\varphi(x)\cdot_B\varphi(y),\label{10}\\
  \varphi[x,y]_A&=&[\varphi(x),\varphi(y)]_B,\quad \forall~x,y\in A.
\end{eqnarray}
\end{defi}
If Eq.\eqref{10} is not satisfied, then we call $\varphi$ is a weak homomorphism.
\begin{ex}
  Let $(A,\cdot,[\cdot,\cdot])$ be a $F$-manifold algebra and $\a:A\to A$ be a $F$-manifold algebra morphism. Then $(A,\cdot_\a,[\cdot,\cdot]_\a,\a)$ is a Hom-$F$-manifold algebra, where for all $x,y\in A$
  $$
   x\cdot_\a y=\a(x)\cdot \a(y),~~~~~~~[x,y]_\a=[\a(x),\a(y)].
  $$
\end{ex}

\begin{ex}
Any Hom-Poisson algebra is a  Hom-$F$-manifold algebra.

\end{ex}

\begin{ex}\label{ex:direct sum of HMA}{\rm
  Let $(A,\cdot_A,[\cdot,\cdot]_A,\a_A)$ and $(B,\cdot_B,[\cdot,\cdot]_B,\a_B)$ be two Hom-$F$-manifold algebras. Then $(A\oplus B,\cdot_{A\oplus B},[\cdot,\cdot]_{A\oplus B},\a_{A\oplus B})$ is a Hom-$F$-manifold algebra, where the product $\cdot_{A\oplus B}$, the bracket $[\cdot,\cdot]_{A\oplus B}$ and the twist map $\a_{A\oplus B}$ are given by
  \begin{eqnarray*}
   ( x_1+ x_2) \cdot_{A\oplus B} (y_1+ y_2)&=& x_1\cdot_A y_1+x_2\cdot_B y_2,\\
   {[  ( x_1+ x_2), (y_1+ y_2)]_{A\oplus B}}&=&[x_1,y_1]_A+ [x_2, y_2]_B,\\
   \a_{A\oplus B}&=&\a_A+\a_B
  \end{eqnarray*}
  for all $x_1,y_1\in A,x_2,y_2\in B.$}
\end{ex}

\begin{ex}\label{ex:tensor of HMA}{\rm
  Let $(A,\cdot_A,[\cdot,\cdot]_A,\a_A)$ and $(B,\cdot_B,[\cdot,\cdot]_B,\a_B)$ be two Hom-$F$-manifold algebras. Then $(A\otimes B,\cdot_{A\otimes B},[\cdot,\cdot]_{A\otimes B},\a_{A\otimes B})$ is an $F$-manifold algebra, where the product $\cdot_{A\otimes B}$, the  bracket $[\cdot,\cdot]_{A\otimes B}$ and the twist map $\a_{A\otimes B}$ are given by
  \begin{eqnarray*}
   ( x_1\otimes x_2) \cdot_{A\otimes B} (y_1\otimes y_2)&=& (x_1\cdot_A y_1) \otimes (x_2\cdot_B y_2),\\
   {[ x_1\otimes x_2, y_1\otimes y_2]_{A\otimes B}}&=&[x_1,y_1]_A \otimes (x_2\cdot_B y_2)+(x_1\cdot_A y_1) \otimes [x_2,y_2]_B,\\
\a_{A\otimes B}&=&\a_A\otimes\a_B
  \end{eqnarray*}
  for all $x_1,y_1\in A,x_2,y_2\in B.$}
\end{ex}


Now, we introduce the notion of   Hom-$F$-manifold admissible algebras.
\begin{defi}\label{defi:pseudo-pre-HMA}
A  Hom-$F$-manifold admissible algebra  is a tuple $(A,\cdot,\ast,\a)$ such that $(A,\cdot,\a)$ is a commutative Hom-associative algebra and $(A,\ast,\a)$ is a Hom-Lie admissible algebra satisfying for all $x,y,z\in A$,
\begin{equation}\label{eq:pseudo-pre-HM1}
\a(x)\ast (y\cdot  z)-(x\ast  y)\cdot  \a(z)-\a(y)\cdot (x\ast z)=\a(y)\ast(x\cdot z)-(y\ast x)\cdot \a(z)-\a(x)\cdot(y\ast z).
\end{equation}
\end{defi}

\begin{pro}\label{pro:pseudo-pre-HMA-HMA}
  Let $(A,\cdot,\ast,\a)$ be a Hom-$F$-manifold admissible algebra. Then $(A,\cdot,[\cdot,\cdot],\a)$ is a Hom-$F$ manifold algebra, where the bracket $[\cdot,\cdot]$ is given by
\begin{equation}\label{eq:pseudo-bracket}
[x,y]=x\ast y-y\ast x,\quad\forall~x,y\in A.
\end{equation}

\end{pro}
\begin{proof}
 Let $x,y,z,w\in A$. Using Eq. \eqref{eq:pseudo-pre-HM1}, we have
  \begin{eqnarray*}
    \mathcal L(x,y,z)&=&[\a(x),y\cdot  z] -[x,y] \cdot \a( z)-\a(y)\cdot  [x,z] \\
    &=&\a(x)\ast  (y\cdot  z)-(y\cdot  z)\ast \a( x)-(x\ast  y)\cdot \a( z)+(y\ast  x)\cdot  \a(z)\\
    &&-\a(y)\cdot (x\ast  z)+\a(y)\cdot (z\ast  x)\\
    &=&\a(x)\ast  (y\cdot  z)-(x\ast  y)\cdot  \a(z)-\a(y)\cdot (x\ast  z)-(y\cdot  z)\ast\a(  x)\\
    &&+(y\ast  x)\cdot \a( z)+\a(y)\cdot (z\ast  x)\\
    &=&\a(y)\ast  (x\cdot  z)-(y\ast  x)\cdot\a(  z)-\a(x)\cdot (y\ast  z)-(y\cdot  z)\ast \a( x)\\
    &&+(y\ast  x)\cdot \a( z)+\a(y)\cdot (z\ast  x).
  \end{eqnarray*}
  By this formula and Eq.\eqref{eq:pseudo-pre-HM1}, we have
  \begin{eqnarray*}
   &&\mathcal L(x\cdot  y,\a(z),\a(w))\\
   &=&\a^2(z)\ast  ((x\cdot  y)\cdot  \a(w))-(\a(z)\ast (x\cdot   y))\cdot  \a^2(w)\\
   &&-(\a(x)\cdot  \a(y))\cdot (\a(z)\ast  \a(w))-(\a(z)\cdot  \a(w))\ast  (\a(x)\cdot  \a(y))\\
    &&+(\a(z)\ast  (x\cdot  y))\cdot \a^2( w)+\a^2(z)\cdot (\a(w)\ast  (x\cdot  y))\\
    &=&\a^2(z)\ast  (\a(x)\cdot  (y\cdot  w))-\big((z\ast  x)\cdot  \a(y)-\a(x)\cdot (z\ast  y)+\a(x)\ast (z\cdot  y)-(x\ast  z)\cdot  \a(y)\\
    &&-\a(z)\cdot (x\ast  y)\big)\cdot  \a^2(w)-(\a(x)\cdot  \a(y))\cdot (\a(z)\ast \a(w))\\
    &&-\big(((z\cdot  w)\ast  \a(x))\cdot  \a^2(y)+\a^2(x)\cdot ((z\cdot  w)\ast  \a(y))\\
    &&+\a^2(x)\ast ((z\cdot  w)\cdot  \a(y))-(\a(x)\ast  (z\cdot  w)\cdot \a^2(y))-(\a(z)\cdot \a( w))\cdot  (\a(x)\ast  \a(y))\big)\\
    &&+\big((z\ast  x)\cdot  \a(y)
    +\a(x)\cdot (z\ast  y)+\a(x)\ast (z\cdot  y)-(x\ast  z)\cdot \a(y)-\a(z)\cdot (x\ast  y))\big)\cdot  \a^2(w)\\
    &&+\big((w\ast  x)\cdot  \a(y)+\a(x)\cdot  (w\ast  y)+\a(x)\ast  (w\cdot  y)-(x\ast  w)\cdot \a(y)-\a(w)\cdot (x\ast  y)\big)\cdot  \a^2(z)\\
    &=&\a^2(y)\cdot  \mathcal L(x,z,w)+\a^2(x)\cdot  \big(-\a(w)\cdot  (z\ast  y)-\a(y)\cdot (z\ast  w)-(z\cdot  w)\ast  \a(y)+(z\ast y)\cdot  \a(w)\\
    &&+(w\ast  y)\cdot  \a(z)\big)+\big(\a^2(z)\ast  (\a(x)\cdot  (y\cdot  w))-(\a(z)\ast  \a(x))\cdot  (\a(y)\cdot  \a(w))\\
    &&-\a^2(x)\ast  (\a(z)\cdot  (y\cdot  w))+(\a(x)\ast \a(z))\cdot  (\a(y)\cdot \a(w))+\a^2(z)\cdot (\a(x)\ast (y\cdot  w))\big)\\
    &=&\a^2(y)\cdot  \mathcal L(x,z,w)+\a^2(x)\cdot  \big(-\a(w)\cdot  (z\ast  y)-\a(y)\cdot (z\ast  w)-(z\cdot  w)\ast  \a(y)+(z\ast y)\cdot \a(w)\\
    &&+(w\ast  y)\cdot \a(z)+\big)+\a^2(x)\cdot (\a(z)\ast  (y\cdot  w))\\
    &=&\a^2(y)\cdot  \mathcal L(x,z,w)+\a^2(x)\cdot \mathcal L(y,z,w).
  \end{eqnarray*}
Then, the  Hom-Hertling-Manin identity \eqref{eq:HM relation} holds.

\end{proof}
As a special case of Hom-$F$-manifold  admissible algebras, we define Hom-pre-Lie commutative algebras.
\begin{defi}
A { Hom-pre-Lie commutative algebra  } is a tuple $(A,\cdot,\ast,\a)$, where $(A,\cdot,\a)$ is a commutative Hom-associative algebra and $(A,\ast,\a)$ is a Hom-pre-Lie algebra satisfying
\begin{equation}
  \a(x)\ast(y\cdot z)-(x\ast y)\cdot\a( z)-\a(y)\cdot(x\ast z)=0,\quad\forall~x,y,z\in A.
\end{equation}
\end{defi}
Using Proposition \ref{pro:pseudo-pre-HMA-HMA}, it is  obvious to obtain the following result.
\begin{cor}
  Let $(A,\cdot,\ast,\a)$ be a Hom-pre-Lie commutative algebra. Then $(A,\cdot,[\cdot,\cdot],\a)$ is a Hom-$F$-manifold algebra,
  where $[\cdot,\cdot]$ is given by \eqref{eq:pseudo-bracket}.
\end{cor}

Recall that a derivation of commutative Hom-associative algebra $(A,\cdot,\a)$ is a linear map $D:A\to A$ such that
\begin{align*}
    D \alpha&=\alpha  D\\
    D(x\cdot y)&=D(x)\cdot y+x\cdot D(y),\forall x,y\in A.
\end{align*}

\begin{pro}\label{ex:derivation-HMA}
 Let $(A,\cdot,\a)$ be a commutative Hom-associative algebra with a derivation $D$. Then $(A,\cdot,\ast,\a )$ being a Hom-$F$-manifold admissible algebra where the new product $\ast$ is defined by
  \begin{eqnarray*}
    x\ast  y&=&x\cdot D (y)+\lambda x\cdot y,\quad\forall~x,y\in A
  \end{eqnarray*}
     for any fixed $\lambda\in \K$. In particular, for $\lambda=0$, $(A,\cdot,\ast,\a)$ is a Hom-pre-Lie commutative algebra.
\end{pro}
\begin{proof}
  It is easy to show that $(A,\ast ,\a)$ is a Hom-pre-Lie algebra. Furthermore, by the fact that $(A,\cdot,\alpha)$ is a commutative Hom-associative algebra and  $D$ is a derivation on it, we have
  \begin{eqnarray*}
   && \a(x)\ast (y\cdot z)-(x\ast  y)\cdot \a(z)-\a(y)\cdot(x\ast  z)\\
    &=&\a(x)\cdot D (y\cdot z)+\lambda\a(x)\cdot( y\cdot z)-(x\cdot D (y)+\lambda x\cdot y)\cdot \a( z)-\a(y)\cdot (x\cdot D (z)+\lambda x\cdot z)\\
    &=&-\lambda \a(x)\cdot( y\cdot z).
  \end{eqnarray*}
  Similarly, we have
  $$\a(y)\ast (x\cdot z)-(y\ast  x)\cdot \a(z)-\a(x)\cdot(y\ast  z)=-\lambda \a(x)\cdot( y\cdot z).$$
  Thus $$\a(x)\ast (y\cdot z)-(x\ast  y)\cdot \a(z)-\a(y)\cdot(x\ast  z)=\a(y)\ast (x\cdot z)-(y\ast  x)\cdot \a(z)-\a(x)\cdot(y\ast  z).$$
  Therefore,  $(A,\cdot,\ast ,\a)$ is a Hom-$F$-manifold admissible algebra. When $\lambda=0$, it is obvious that $(A,\cdot,\ast,\a)$ is a Hom-pre-Lie commutative algebra.
\end{proof}

\begin{ex}{
  Let $A$ be a $2$-dimensional vector space  with basis $\{e_1,e_2\}$.  Define the non-zero multiplication by
  \begin{eqnarray*}
  e_1\cdot e_1&=&e_1,\quad e_1\cdot e_2=e_2\cdot e_1=be_2,\quad~b\in \K
  \end{eqnarray*}
  and the linear map $\alpha:A\to A$   by
  $$\alpha(e_1)=e_1, ~~\alpha(e_2)=be_2.$$
 Then $(A,\cdot,\alpha)$  is  commutative Hom-associative algebra.
  It is straightforward to check that the linear map $D:A\to A$ given by
  $$D(e_2)=a e_2,\quad~a\in \K$$ is a derivation on $(A,\cdot,\alpha)$.
  Thus by Proposition \ref{ex:derivation-HMA}, $(A,\cdot,[\cdot,\cdot],\alpha)$
  is a Hom-$F$-manifold algebra, where
  $[e_1,e_2]=a be_2$.

  }
\end{ex}

\begin{ex}\label{ex:3-dimensional F-algebra}{\rm
  Let $A$ be a $3$-dimensional vector space  with basis $\{e_1,e_2,e_3\}$. Define the non-zero multiplication by
  \begin{eqnarray*}
 e_2\cdot e_3=e_3\cdot e_2=b^3e_1,\quad  e_3\cdot e_3=b^2e_2,\quad ~b\in \K
  \end{eqnarray*}
   and the linear map $\alpha:A\to A$   by
  $$\alpha(e_1)=b^3e_1, ~~\alpha(e_2)=b^2e_2, ~~\alpha(e_3)=be_3.$$
 Then $(A,\cdot,\alpha)$  is  commutative Hom-associative algebra.
  It is straightforward to check that the linear map $D:A\to A$ given by
 \begin{eqnarray*}
  D(e_1)=3  e_1,\quad D(e_2)=2  e_2,\quad
  D(e_3)= e_3
  \end{eqnarray*}
  is a derivation on $(A,\cdot,\alpha)$.
  Thus by Proposition \ref{ex:derivation-HMA}, $(A,\cdot,[\cdot,\cdot],\alpha)$
  is a Hom-$F$-manifold algebra, where
  $[e_2,e_3]=-  b^3e_1$.
  }
\end{ex}

\section{Representations of Hom-$F$-manifold algebras and Hom-pre-$F$-manifold algebras}\label{sec:pre-F-algebras}
In this section, first we study representations of a Hom-$F$-manifold algebra. Then we introduce the notions of Hom-pre-$F$-manifold
algebras and   $\huaO$-operators on a Hom-$F$-manifold algebra.

\subsection{Representations of Hom-$F$-manifold algebras}

Let $(A,\cdot ,[\cdot,\cdot],\a)$ be a Hom-$F$-manifold algebra, $(V;\rho,\phi)$ be a
  representation of the Hom-Lie algebra $(A,[\cdot,\cdot],\a)$ and $(V;\mu,\phi)$ be a representation of the commutative Hom-associative algebra $(A,\cdot,\a)$. Define the three linear maps $\mathcal{L}_1,\mathcal{L}_2, \mathcal{L}_3:A\otimes A\otimes V\to V$
  given by
  \begin{eqnarray}
  \label{eq:repH 1}\mathcal{L}_1(x,y,u)&=&\rho(\a(x))\mu(y)(u)-\mu(\a(y)) \rho(x)(u)-\mu([x,y])\phi(u),\\
  \label{eq:repH 2}\mathcal{L}_2(x,y,u)&=&\mu(\a(x))\rho(y)(u)+\mu(\a(y)) \rho(x)(u)-\rho(x\cdot y)\phi(u),\\
\label{eq:repH 3} \mathcal{L}_3(x,y,u)&=&\rho(\a(y))  \mu(x)(u)+\rho(\a(x))  \mu(y)(u)-\rho(x\cdot  y)  \phi(u)
   \end{eqnarray}
  for all $x,y\in A$ and $u\in V$. Note that, if we consider $V=A$, then $\mathcal{L}_1=\mathcal{L}$ and $\mathcal{L}_2=\mathcal{L}\sigma$, where $\sigma(x,y,z)=(z,x,y)$.
\begin{defi}\label{def-rep-manifo}With the above notations, the tuple $(V;\rho,\mu,\phi)$ is a representation of $A$  if the following conditions hold:
   \begin{eqnarray}
     \label{eq:rep 1}\mathcal{L}_1(x\cdot  y,\a(z),\phi(u))=\mu(\a^2(x)) \mathcal{L}_1(y,z,u)+\mu(\a^2(y)) \mathcal{L}_1(x,z,u),\\
     \label{eq:rep 2} \mu(\mathcal L(x,y,z))\phi^2(u)=\mathcal{L}_2(\a(y),\a(z),\mu(x)(u))-\mu(\a^2(x)) \mathcal{L}_2(y,z,u)
   \end{eqnarray}
  for all $x,y,z\in A$ and $u\in V$.
  \end{defi}
  \begin{ex}
  Let $(A,\cdot ,[\cdot,\cdot],\a)$ be a Hom-$F$-manifold algebra. Then $(A,\ad,L,\alpha)$ is a representation of $A$ called the adjoint representation.
  \end{ex}
\begin{ex}
 Let $(V;\rho,\mu,\phi)$ be a representation of a Hom-Poisson algebra $(P,\cdot,[\cdot,\cdot],\alpha)$, i.e. $(V;\rho,\phi)$ is a  representation of the Hom-Lie
 algebra $(P,[\cdot,\cdot],\a)$ and $(V;\mu,\phi)$ is a representation of the commutative Hom-associative algebra $(P,\cdot,\a)$ satisfying
   \begin{eqnarray*}
          &&\mathcal{L}_1(x,y,u)=\mathcal{L}_2(x,y,u)  =0,\quad \forall~x,y,z\in P.
   \end{eqnarray*}
  Then $(V;\rho,\mu,\phi)$ is also a representation of the Hom-$F$-manifold algebra given by this Hom-Poisson algebra $P$.
\end{ex}

\begin{ex}
Let $(V;\rho,\mu)$ be a representation of a $F$-manifold algebra $(A,\cdot ,[\cdot,\cdot])$,  $\a:A\to A$ be an algebra morphism  and $\phi\in \mathfrak{gl}(V)$ such that for all $x\in A$
\begin{align*}
   \phi \rho(x)&=\rho(\alpha(x))\phi,\\
       \phi \mu(x)&=\mu(\alpha(x))\phi.
\end{align*}
Then  $(V;\tilde \rho,\tilde \mu,\phi)$  is   a representation of the Hom-$F$-manifold $(A,\cdot_\a ,[\cdot,\cdot]_\a,\a)$, where \begin{align*}
   \tilde \rho(x)=\rho(\alpha(x))\phi~~\ \ \  \text{and}~~\ \ \
       \tilde \mu(x)=\mu(\alpha(x))\phi.
\end{align*}
\end{ex}
It obvious to obtain the following result.
\begin{pro}\label{pro:semi-direct}
 Let $(A,\cdot ,[\cdot,\cdot],\a)$ be a Hom-$F$-manifold algebra. Then $(V;\rho,\mu,\phi)$ is a representation  of $A$ if and only if $(A\oplus V,\cdot_{\mu},[\cdot,\cdot]_\rho,\a+\phi)$ is
 a  Hom-$F$-manifold algebra, where $(A\oplus V,\cdot_{\mu},\alpha+\phi)$ is the semi-direct product commutative Hom-associative algebra $A\ltimes_{\mu} V$, i.e.
 $$
  (x_1+v_1)\cdot_{\mu}(x_2+v_2)=x_1\cdot  x_2+\mu(x_1)v_2+\mu(x_2)v_1,\quad \forall~ x_1,x_2\in A,~v_1,v_2\in V
$$
$$
(\alpha+\phi)(x_1+v_1)=\a(x_1)+\phi(v_1),~~~\forall x_1\in A, ~~ v_1\in V,
$$
  and $(A\oplus V,[\cdot,\cdot]_\rho,\alpha+\phi)$ is the semi-direct product Hom-Lie algebra $A\ltimes_{\rho} V$, i.e. $$
  [x_1+v_1,x_2+v_2]_\rho=[x_1,x_2]+\rho(x_1)(v_2)-\rho(x_2)(v_1),\quad \forall~x_1,x_2\in A,~v_1,v_2\in V.
$$
\end{pro}

\begin{rmk}
Let $(V;\rho,\mu,\phi)$ be a representation of a Hom-Poisson algebra $(P,\cdot,[\cdot,\cdot],\a)$. Then the tuple $(V^*;\rho^\star,-\mu^\star,(\phi^*)^{-1})$
is also a representation of $P$. But  Hom-$F$-manifold algebras do not have this property.
\end{rmk}
\begin{lem}
  Let $(A,\cdot ,[\cdot,\cdot],\a)$ be a Hom-$F$-manifold algebra and     $(V;\rho,\mu,\phi)$ be a representation of $A$.  Define the two linear maps $\mathcal{L}_1^\star,\mathcal{L}_2^\star:A\otimes A\otimes V^*\to V^*$
  given by
  \begin{eqnarray*}
  \label{eq:repH 1}\mathcal{L}_1^\star(x,y,\xi)&=&-\rho^\star(\a(x))\mu^\star(y)(\xi)+\mu^\star(\a(y)) \rho^\star(x)(\xi)+\mu^\star([x,y])(\phi^{-1})^*(\xi),\\
  \label{eq:repH 2}\mathcal{L}_2^\star(x,y,\xi)&=&-\mu^\star(\a(x))\rho^\star(y)(\xi)-\mu^\star(\a(y)) \rho^\star(x)(\xi)-\rho^\star(x\cdot y)(\phi^{-1})^*(\xi)
   \end{eqnarray*}
  for all $x,y\in A$ and $\xi\in V^*$. Then we have
  $$\langle\mathcal{L}_1^\star(x,y,\xi),u\rangle=\langle\xi,\mathcal{L}_1(\a^{-2}(x),\a^{-2}(y),\phi^{-4}(u))\rangle$$
  and
  $$\langle\mathcal{L}_2^\star(x,y,\xi),u\rangle=-\langle\xi,\mathcal{L}_3(\a^{-2}(x),\a^{-2}(y),\phi^{-4}(u))\rangle.$$
  for all $x,y\in A$ and $u\in V$.
\end{lem}
    \begin{proof}
        Straightforward.
    \end{proof}
\begin{pro}
  Let $(A,\cdot ,[\cdot,\cdot],\a)$ be a Hom-$F$-manifold algebra. If    the tuple $(V;\rho,\mu,\phi)$ is representation of $A$ satisfying the following identities
  \begin{eqnarray}
     \label{eq:corep 1}\mathcal{L}_1(x\cdot  y,z,\phi(u))= \mathcal{L}_1(\a(y),z,  \mu(x)(u)))+ \mathcal{L}_1(\a(x),z,  \mu(y)(u)),\\
     \label{eq:corep 2} \mu(\mathcal L(x,y,z))  \phi^2(u)=\mathcal{L}_3(\a(y),\a(z),  \mu(x)(u))-\mu(\a(x))  \mathcal{L}_3(y,z,u).
   \end{eqnarray}
   Then $(V^*;\rho^\star,-\mu^\star,(\phi^{-1})^*)$ is  a representation of $A$.
\end{pro}
\begin{proof}
By direct calculations, for all $x,y,z\in A,v\in V,\xi\in V^*$, we have
\begin{eqnarray*}
 &&\langle\mathcal{L}_1^\star(x\cdot  y,\a(z),(\phi^{-1})^*(\xi))+\mu^\star(\a^2(x)) \mathcal{L}_1^\star(y,z,\xi)+\mu^\star(\a^2(y)) \mathcal{L}_1^\star(x,z,\xi),v\rangle\\
 &=&\langle \xi,\mathcal{L}_1(\a^{-3}(x)\cdot  \a^{-3}(y),\a^{-2}(z),\phi^{-5}(v))- \mathcal{L}_1(\a^{-2}(y),\a^{-2}(z),\mu(\a^{-3}(x))(\phi^{-6}(v)))\\&&- \mathcal{L}_1(\a^{-2}(x),\a^{-2}(z),\mu(\a^{-3}(y))(\phi^{-6}(v))\rangle
\end{eqnarray*}
and
\begin{eqnarray*}
 &&\langle -\mu^\star(\mathcal L(x,y,z))(\phi^{-2})^*(\xi)+\mathcal{L}_2^\star(\a(y),\a(z),\mu^\star(x)(\xi))-\mu^\star(\a^2(x)) \mathcal{L}_2^\star(y,z,\xi),v\rangle\\
 &=&\langle \xi,\mu(\mathcal L(\a^{-3}(x),\a^{-3}(y),\a^{-3}(z)))(\phi^{-4}(v))+\mu(\a^{-2}(x))\mathcal{L}_3(\a^{-3}(y),\a^{-3}(z),\phi^{-6}(u))\\&&- \mathcal{L}_3(\a^{-2}(y),\a^{-2}(z),\mu(\a^{-3}(x))\phi^{-6}(v)\rangle.
\end{eqnarray*}
By Eqs. \eqref{eq:corep 1}, \eqref{eq:corep 2} and the Definition \ref{def-rep-manifo}, the conclusion follows immediately.
\end{proof}

\begin{cor}\label{ex:dual representation}{\rm
  Let $(A,\cdot ,[\cdot,\cdot],\alpha )$ be a Hom-$F$-manifold algebra such that  the following relations hold:
  \begin{eqnarray}
   \label{eq:coh1} \mathcal{L} (x\cdot  y,z,\alpha(w))&=&\mathcal{L} (\a(y),z,x\cdot  w)+\mathcal{L} (\a(x),z,y\cdot  w),\\
    \label{eq:coh2} \mathcal{L} (x, y,z)\cdot \alpha^2(w)&=&\mathcal K(\alpha(y),\alpha(z),x\cdot  w)-\a(x) \cdot  \mathcal K(y,z,w)
  \end{eqnarray}
   for all $x,y,z,w\in A$,
 where $\mathcal K:\otimes^3 A\rightarrow
A$ is defined by
$$\mathcal K(x,y,z)=\circlearrowleft_{x,y,z}[\alpha(x),y\cdot  z],\ \ \forall\ x,\ y,\ z\in A.$$
Then $(A^*;\ad^\star,-L^\star,(\alpha^{-1})^*)$ is a representation of $A$ called the coadjoint representation.}
\end{cor}

\begin{defi}
  A { coherence Hom-$F$-manifold algebra} is a Hom-$F$-manifold algebra such that Eqs.\eqref{eq:coh1} and  \eqref{eq:coh2} hold.
\end{defi}

\subsection{Hom-pre-$F$-manifold algebras}

Now,  we introduce the notion of Hom-pre-$F$-manifold algebras and give some constructions.
\begin{defi}
  A {  Hom-pre-$F$-manifold algebra} is a tuple $(A,\diamond,\ast,\alpha)$, where $(A,\diamond,\alpha)$ is a Hom-Zinbiel algebra and $(A,\ast,\alpha)$ is a Hom-pre-Lie algebra, such that  the following compatibility conditions hold:
  \begin{align}
    \label{eq:pre-HM 1}&F_1(x\cdot y,\alpha(z),\alpha(w))=\alpha^2(x)\diamond F_1(y,z,w)+\alpha^2(y)\diamond F_1(x,z,w),\\
    \label{eq:pre-HM 2}&\mathcal L(x,y,z)\diamond \alpha^2(w)=F_2(\alpha(y),\alpha(z),x\diamond w)-\alpha^2(x)\diamond F_2(y,z,w)
  \end{align}
 where $F_1,F_2,\mathcal L:\otimes^3 A\longrightarrow A$ are defined by
  \begin{eqnarray}
    F_1(x,y,z)&=&\alpha(x)\ast(y\diamond z)-\alpha(y)\diamond(x\ast z)-[x,y]\diamond \alpha(z),\\
    F_2(x,y,z)&=&\alpha(x)\diamond(y\ast z)+\alpha(y)\diamond(x\ast z)-(x\cdot y)\ast \alpha(z),\\
     \mathcal L(x,y,z)&=&[\a(x),y\cdot z]-[x,y]\cdot \a(z)-\a(y)\cdot [x,z]
  \end{eqnarray}
and the operation $\cdot$ and bracket $[\cdot,\cdot]$ are defined by
\begin{equation}\label{eq:pHM-operations}
  x\cdot y=x\diamond y+y\diamond x,\quad [x,y]=x\ast y-y\ast x,
\end{equation}
for all $x,y,z,w\in A$.
\end{defi}

\begin{rmk}
  If $F_1=F_2=0$ in the above definition of a Hom-pre-$F$-manifold algebra $(A,\diamond,\ast,\alpha)$, then we obtain a  Hom-pre-Poisson algebra ( See \cite{Aguiar,Guo&Zhang&Wang} for more details).
\end{rmk}

\begin{thm}
  Let $(A,\diamond,\ast,\alpha)$ be a Hom-pre-$F$-manifold algebra. Then
   \begin{itemize}
\item[$\rm(i)$]
   $(A,\cdot,[\cdot,\cdot],\alpha)$ is a  Hom-$F$-manifold algebra, where the operation $\cdot$ and bracket $[\cdot,\cdot]$ are given by Eq. \eqref{eq:pHM-operations}, which is called the   sub-adjacent
Hom-$F$-manifold algebra  of $(A,\diamond,\ast,\alpha)$  and denoted by $A^c$.
\item[$\rm(ii)$]$(A;L_\ast,L_\diamond,\alpha)$ is a representation of the sub-adjacent
Hom-$F$-manifold algebras $A^c$, where $L_\ast$ and $L_\diamond$ are given by Eqs. \eqref{eq:defiLpreLie} and \eqref{eq:dendriform-rep}, respectively.
\end{itemize}
\end{thm}
\begin{proof}
\begin{itemize}
    \item[(i)]
 By Lemma  \ref{lem:den-ass} and Lemma \ref{lem:pre-Lie-Lie},  we deduce that $(A,\cdot,\alpha)$ is a commutative Hom-associative algebra and $(A,[\cdot,\cdot],\alpha)$ is a Hom-Lie algebra.  By direct computation, we obtain
  \begin{equation}\label{eq:pre-HM 3}
 \mathcal L(x,y,z)=F_1(x,y,z)+F_1(x,z,y)+F_2(y,z,x),\quad\forall~x,y,z\in A.
  \end{equation}

  According to Eqs.\eqref{eq:pre-HM 1}, \eqref{eq:pre-HM 2} and \eqref{eq:pre-HM 3}, we get
  \begin{eqnarray*}
    &&\mathcal L(x\cdot y,\alpha(z),\alpha(w))-\alpha^2(x)\cdot \mathcal L(y,z,w)-\alpha^2(y)\cdot \mathcal L(x,z,w)\\
&=&\big( F_1(x\cdot y,\alpha(z),\alpha(w))-\alpha^2(x)\diamond F_1(y,z,w)-\alpha^2(y)\diamond F_1(x,z,w)\big)\\
    &&+\big( F_1(x\cdot y,\alpha(w),\alpha(z))-\alpha^2(x)\diamond F_1(y,w,z)-\alpha^2(y)\diamond F_1(x,w,z)\big)\\
    &&+\big(F_2(\alpha(z),\alpha(w),x\diamond y)-\mathcal L(x,z,w)\diamond \alpha^2(y)-\alpha^2(x)\diamond F_2(z,w,y)\big)\\
    &&+\big(F_2(\alpha(z),\alpha(w),y\diamond x)-\mathcal L(y,z,w)\diamond \alpha^2(x)-\alpha^2(y)\diamond F_2(z,w,x)\big)\\
    &=& 0.
  \end{eqnarray*}
  Thus $(A,\cdot,[\cdot,\cdot],\alpha)$ is a Hom-$F$-manifold algebra.

\item[(ii)] Thanks to  Lemma  \ref{lem:den-ass} and Lemma \ref{lem:pre-Lie-Lie},  $(A;L_\diamond,\alpha)$ is a representation of the commutative Hom-associative algebra $(A,\cdot,\alpha)$ as well as $(A;L_\ast,\alpha)$ is a representation of the sub-adjacent Hom-Lie algebra $A^c$. According to  Eqs. \eqref{eq:pre-HM 1}  and \eqref{eq:pre-HM 2} we can easly check  Eq. \eqref{eq:rep 1} and Eq.\eqref{eq:rep 2}.  Therefore,  $(A;L_\ast,L_{\diamond},\alpha)$ is a representation of the sub-adjacent Hom-$F$-manifold algebra $A^c$.
\end{itemize}
\end{proof}

\subsection{ $\huaO$-operators of Hom-$F$-manifold algebras}

The notion of an $\mathcal O$-operator was first given for Lie algebras by Kupershmidt in \cite{K} as a natural generalization of the classical Yang-Baxter equation and then defined by analogy in other various (associative, alternative, Jordan, ...).

A linear map $T:V\longrightarrow A$ is called an { $\huaO$-operator} on a commutative Hom-associative algebra $(A,\cdot,\alpha)$ with respect to a representation  $(V;\mu,\phi)$ if $T$ satisfies
  \begin{align}
    T \phi  &= \alpha T,\\
    T(u)\cdot  T(v)&=T(\mu(T(u))v+\mu(T(v))u),\quad\forall~u,v\in V.
  \end{align}
In particular, an $\huaO$-operator on a commutative Hom-associative
algebra $(A,\cdot,\alpha )$ with respect to the adjoint representation
is called a { Rota-Baxter operator of weight zero} or briefly a { Rota-Baxter
operator} on $A$.

 It is obvious to obtain the following result.
\begin{pro}
  Let $(A,\cdot,\alpha )$ be a commutative Hom-associative algebra and $(V;\mu,\phi)$ a representation. Let $T:V\rightarrow A$ be an $\huaO$-operator on  $(A,\cdot,\alpha )$ with respect to $(V;\mu,\phi)$. Then there exists a Hom-Zinbiel algebra structure on $V$ given by
$$
    u\diamond v=\mu(T(u))v, \quad\forall~u,v\in V.
$$
\end{pro}

A linear map $T:V\longrightarrow   A$ is called an { $\huaO$-operator} on a Hom-Lie algebra $(A,[\cdot,\cdot],\alpha)$ with respect to a representation $(V;\rho,\phi)$ if $T$ satisfies
\begin{align}
  T \phi &=  \alpha T,\label{O-operator1}\\
  [T(u), T(v)] &=T\Big(\rho(T(u))(v)-\rho(T(v))(u)\Big),\quad \forall~u,v\in V.\label{O-operator2}
\end{align}
In particular, an $\huaO$-operator on a Hom-Lie algebra
$(  A,[\cdot,\cdot] ,\alpha)$ with respect to the adjoint representation is called a {
Rota-Baxter operator of weight zero} or briefly a {  Rota-Baxter operator} on $  A$.

\begin{lem}{\rm(\cite{Liu&Song&Tang})}
Let $T:V\to   A$ be an $\huaO$-operator  on a Hom-Lie algebra $(  A,[\cdot,\cdot],\alpha )$ with respect to a representation $(V;\rho,\phi)$. Define a multiplication $\ast$ on $V$ by
\begin{equation}
  u\ast v=\rho(T(u))(v),\quad \forall~u,v\in V.
\end{equation}
Then $(V,\ast,\alpha)$ is a Hom-pre-Lie algebra.
 \end{lem}

Let $(V;\rho,\mu,\phi)$ be a representation of a Hom-$F$-manifold algebra $(A,\cdot ,[\cdot,\cdot],\alpha )$.
\begin{defi}
 A linear operator $T:V\longrightarrow A$ is called  an  $\huaO$-operator on $(A,\cdot ,[\cdot,\cdot],\alpha )$ if $T$ is both an $\huaO$-operator on the commutative Hom-associative algebra $(A,\cdot,\alpha )$ and an $\huaO$-operator on the Hom-Lie algebra $(A,[\cdot,\cdot],\alpha )$.\end{defi}
In particular,   a linear operator $\mathcal R:A\longrightarrow A$ is
called a {
Rota-Baxter operator of weight zero} or briefly a { Rota-Baxter
operator} on $A$, if $\mathcal R$ is both a Rota-Baxter operator on
the commutative Hom-associative algebra $(A,\cdot,\alpha )$ and a
Rota-Baxter operator on the Hom-Lie algebra $(A,[\cdot,\cdot],\alpha)$.

In the following result, we give the construction of Hom-pre-$F$-manifold algebra using an $\huaO$-operator on a Hom-$F$-manifold algebra.

\begin{thm}\label{thm:pre-F-algebra and F-algebra}
Let $(A,\cdot ,[\cdot,\cdot],\alpha)$ be a Hom-$F$-manifold algebra and $T:V\longrightarrow
A$ an $\huaO$-operator on  $A$ with respect to the representation $(V;\rho,\mu,\phi)$. Define new operations $\diamond$ and $\ast$ on $V$ by
$$ u\diamond v=\mu(T(u))v,\quad u\ast v=\rho(T(u))v,\ \forall \ u,v\in V.$$
Then $(V,\diamond,\ast,\phi)$ is a Hom-pre-$F$-manifold algebra and $T$ is a homomorphism from $V^c$ to \\$(A,\cdot ,[\cdot,\cdot],\alpha)$.
\end{thm}
\begin{proof}Since $T$ is an $\huaO$-operator on the commutative Hom-associative algebra $(A,\cdot,\alpha )$ as well as an $\huaO$-operator on the Hom-Lie algebra $(A,[\cdot,\cdot],\alpha )$ with respect to the representations $(V;\mu, \phi)$ and $(V;\rho,\phi)$ respectively. We deduce that $(V,\diamond,\phi)$ is a Hom-Zinbiel algebra and $(V,\ast,\phi)$ is a Hom-pre-Lie algebra.
Put
$$[u,v]_T:=u\ast v-v\ast u=\rho(T(u))v-\rho(T(v))u\ \ \text{and}\ \ \ u\cdot_T v:=u\diamond v+v\diamond u=\mu(T(u))v+\mu(T(v))u.$$

By these facts and Eq.\eqref{eq:rep 1}, for $v_1,v_2,v_3,v_4\in V$, one has
{\small\begin{eqnarray*}  
  &&F_1(v_1\cdot_T v_2,\phi(v_3),\phi(v_4))-\phi^2(v_1)\diamond F_1(v_2,v_3,v_4)-\phi^2(v_2)\diamond F_1(v_1,v_3,v_4)\\
  &=&\phi(v_1\cdot_T v_2)\ast(\phi(v_3)\diamond \phi(v_4))-\phi(\phi(v_3))\diamond((v_1\cdot_T v_2)\ast \phi(v_4))-[v_1\cdot_T v_2,\phi(v_3)]_T\diamond \phi(\phi(v_4))\\
  &&-\phi^2(v_1)\diamond( \phi(v_2)\ast(v_3\diamond v_4))+\phi^2(v_1)\diamond(\phi(v_3)\diamond(v_2\ast v_4))+\phi^2(v_1)\diamond  ([v_2,v_3]_T\diamond \phi(v_4))\\
  &&-\phi^2(v_2)\diamond( \phi(v_1)\ast(v_3\diamond v_4))+\phi^2(v_2)\diamond(\phi(v_3)\diamond(v_1\ast v_4))+\phi^2(v_2)\diamond ([v_1,v_3]_T\diamond \phi(v_4))\\
  &=&\rho(T(\phi(v_1\cdot_T v_2)) (\mu(T(\phi(v_3))) \phi(v_4))-\mu(T(\phi(\phi(v_3)))(\rho(T(v_1\cdot_T v_2)) \phi(v_4))\\
  &&-\mu(T[v_1\cdot_T v_2,\phi(v_3)]_T) \phi(\phi(v_4))
   -\mu(T(\phi^2(v_1))(\rho(T(\phi(v_2))\mu(T(v_3) v_4)\\
   &&+\mu(T(\phi^2(v_1))(\mu(T(\phi(v_3))(\rho(T(v_2)) v_4))+\mu(T(\phi^2(v_1))\mu(T[v_2,v_3]_T) \phi(v_4))\\
  &&-\mu(T(\phi^2(v_2))(\rho(T(\phi(v_1))(\mu(T(v_3) v_4))+\mu(T(\phi^2(v_2))(\mu(T(\phi(v_3))(\rho(T(v_1) v_4))\\
  &&+\mu(T(\phi^2(v_2)) (\mu(T[v_1,v_3]_T) \phi(v_4))\\
  &=&\mathcal{L}_1(T(v_1)\cdot  T(v_2),\a(T(v_3)),\phi(v_4))-\mu(\a^2(T(v_1))) \mathcal{L}_1(T(v_2),T(v_3),v_4)\\
  &&-\mu(\a^2(T(v_2))) \mathcal{L}_1(T(v_1),T(v_3),v_4)=0,
\end{eqnarray*}}
which implies that Eq.\eqref{eq:pre-HM 1} holds.

Similarly, according to Eq.\eqref{eq:rep 2}  we can check Eq.\eqref{eq:pre-HM 2}. Thus, $(V,\diamond,\ast,\phi)$ is a  Hom-pre-$F$-manifold algebra. Furthermore, for all $u,v\in V$ we have
$$
[T(u), T(v)] =T\Big(\rho(T(u))(v)-\rho(T(v))(u)\Big)=T(u\ast v-v\ast u)=T([u,v]_T)
$$
Similarly, it is obvious to check that $T(u)\cdot T(v)=T(u\cdot_T v)$.
Thus, $T$ is a homomorphism from $V^c$ to $(A,\cdot ,[\cdot,\cdot],\alpha)$.
\end{proof}

\begin{cor}
Let $(A,\cdot ,[\cdot,\cdot],\alpha)$ be a Hom-$F$-manifold algebra and $T:V\longrightarrow
A$ an $\huaO$-operator on  $A$ with respect to the representation $(V;\rho,\mu,\phi)$. Then $T(V)=\{T(v)\mid v\in V\}\subset A$ is a subalgebra of $A$ and there is an induced Hom-pre-$F$-manifold algebra structure on $T(V)$ given by
$$T(u)\diamond T(v)=T(u\diamond v),\quad T(u)\ast T(v)=T(u\ast v)$$
for all $u,v\in V$.
\end{cor}

\begin{cor}
 Let $(A,\cdot ,[\cdot,\cdot],\alpha )$ be a Hom-$F$-manifold algebra and $\mathcal R:A\longrightarrow
A$ a Rota-Baxter operator of weight $0$. Define new operations on $A$ by
$$x\diamond y=\mathcal R(x)\cdot  y,\quad x\ast y=[\mathcal R(x),y] .$$
Then $(A,\diamond,\ast,\alpha)$ is a Hom-pre-$F$-manifold algebra and $\mathcal R$ is a homomorphism from the sub-adjacent Hom-$F$-manifold algebras $(A,\cdot_\mathcal R,[\cdot,\cdot]_\mathcal R,\alpha)$ to $(A,\cdot ,[\cdot,\cdot],\alpha )$, where $x\cdot_\mathcal R y=x\diamond y+y\diamond x$ and $[x,y]_\mathcal R=x\ast y-y\ast x$.

\end{cor}

\begin{cor}
  Let  $(A,\cdot,[\cdot,\cdot],\alpha)$ be the Hom-$F$-manifold algebra (defined in   Proposition \ref{ex:derivation-HMA}). If $\mathcal R:A\longrightarrow A$ is a Rota-Baxter operator with weight $0$ on the commutative Hom-associative algebra $(A,\cdot,\a)$  satisfying $\mathcal R  D =D  B$. Then $\mathcal R:A\longrightarrow A$ is a Rota-Baxter operator with weight $0$ on the Hom-Lie algebra $(A,[\cdot,\cdot],\a)$.
 Furthermore, $(A,\diamond_\mathcal R,\ast_\mathcal R,\a)$ is a Hom-pre-$F$-manifold algebra, where
 $$x\diamond_\mathcal R y=\mathcal R(x)\cdot y,\quad x\ast_\mathcal R y=\mathcal R(x)\cdot D(y)-y\cdot D(\mathcal R(x)).$$
 In addition,  the sub-adjacent Hom-$F$-manifold algebra $(A,\cdot_\mathcal R,[\cdot,\cdot]_\mathcal R)$  is given by
 $$x\cdot_\mathcal R y=x\diamond_\mathcal R y+y\diamond_\mathcal R x,\quad [x,y]_\mathcal R=x\ast_\mathcal R y-y\ast_\mathcal R x,~~~\forall ~~ x,y\in A.$$

\end{cor}

\begin{ex}{\rm
  We put $A=C^\infty([0,1])$. Then $(A,\cdot,[\cdot,\cdot],\alpha)$ is an Hom-$F$-manifold algebra where the  multiplication,  the  bracket and the twist map are defined by:
  \begin{eqnarray*}
  f\cdot g&=& \lambda fg,\\
{[f,g]}&=&\lambda(fg'-gf')\\
\alpha(f)&=&\lambda f,\quad\forall~f,g\in A, \ \lambda\in \mathbb R.
  \end{eqnarray*}
  It is well-known that the integral operator is a Rota-Baxter operator with weight $0$:
$$
   \mathcal R:A\rightarrow A,\quad \mathcal R(f)(x):=\int^x_0f(t)dt.
$$
It is easy to see that
$$\mathcal R  \partial_x=\partial_x  \mathcal R=\Id.$$
Thus $(A,\diamond_\mathcal R,\ast_\mathcal R, \alpha)$ is a Hom-pre-$F$-manifold algebra, where
 $$f\diamond_\mathcal R g=g\int^x_0f(t)dt,\quad f\ast_\mathcal R g=g'\int^x_0f(t)dt-f\cdot g$$
 and $(A,\cdot_\mathcal R,[\cdot,\cdot]_\mathcal R, \alpha)$ is the sub-adjacent Hom-$F$-manifold algebra with
 $$f\cdot_\mathcal R g=f\int^x_0g(t)dt+g\int^x_0f(t)dt,\quad [f,g]_\mathcal R=f'\int^x_0g(t)dt -g'\int^x_0f(t)dt.$$

}
\end{ex}

\begin{ex}{\rm
Consider the Hom-$F$-manifold algebra $(A,\cdot,[\cdot,\cdot])$ given by Example \ref{ex:3-dimensional F-algebra}. Define the map $ \mathcal R:A\to A$ by
\begin{eqnarray*}
  \mathcal R(e_1)=\frac{1}{3}e_1,\quad \mathcal R(e_2)=\frac{1}{2}e_2,\quad
  \mathcal R(e_3)=e_3.
\end{eqnarray*}
 It is obivous to check that $\mathcal R$  is a Rota-Baxter operator with weight $0$ on the Hom-$F$-manifold algebra $A$. Thus $(A,\diamond_\mathcal R,\ast_\mathcal R,\a)$ is a Hom-pre-$F$-manifold algebra, where
 \begin{eqnarray*}
 e_2\diamond_\mathcal R e_3&=&\frac{1}{2}b^3e_1,\quad e_3\diamond_\mathcal R e_2=b^3e_1,\quad e_3\diamond_\mathcal R e_3=b^2 e_1;\\
 e_2\ast_\mathcal R e_3&=&-\frac{1}{2}b^3e_1,\quad e_3\ast_\mathcal R e_2=b^3e_1,\quad \quad \quad \quad \quad
 \end{eqnarray*}
 and $(A,\cdot_\mathcal R,[\cdot,\cdot]_\mathcal R)$ is the sub-adjacent Hom-$F$-manifold algebra with
\begin{eqnarray*}
e_2\cdot_\mathcal R e_3&=&\frac{3}{2}b^3e_1,\quad e_3\cdot_\mathcal R e_3=2 b^2e_1;\\
{[e_2,e_3]_\mathcal R}&=&-\frac{3}{2}b^3e_1.
\end{eqnarray*}

}
\end{ex}

Now,  we give a necessary and sufficient condition on an Hom-$F$-manifold algebra admitting a
Hom-pre-$F$-manifold algebra structure.

\begin{pro}\label{pro:nsc}
 Let $(A,\cdot ,[\cdot,\cdot] ,\alpha)$ be an Hom-$F$-manifold algebra. There is a Hom-pre-$F$-manifold algebra structure on $A$
 such that its sub-adjacent Hom-$F$-manifold algebra is exactly $(A,\cdot ,[\cdot,\cdot],\alpha)$ if and only if there exists an invertible $\huaO$-operator on $(A,\cdot ,[\cdot,\cdot],\alpha )$.
\end{pro}

\begin{proof} If $T:V\longrightarrow
A$ is an invertible $\huaO$-operator on  $A$ with respect to the representation $(V;\rho,\mu,\phi)$, then the compatible Hom-pre-$F$-manifold algebra structure on $A$ is given by
 $$x\diamond y=T(\mu(x)(T^{-1}(y))),\quad x\ast y=T(\rho(x)(T^{-1}(y)))$$
for all $x,y\in A$, since
\begin{align*} x\diamond y+y\diamond x&=T(\mu(x)(T^{-1}(y)))+T(\mu(y)(T^{-1}(x)))\\
&=T(\mu(TT^{-1}(x))(T^{-1}(y))+\mu(TT^{-1}(y))(T^{-1}(x)))\\
&=TT^{-1}(x)\cdot TT^{-1}(y)=x\cdot y
\end{align*}
and
\begin{align*} x\ast y-y\ast x&=T(\rho(x)(T^{-1}(y)))-T(\rho(y)(T^{-1}(x)))\\
&=T(\rho(TT^{-1}(x))(T^{-1}(y))-\rho(TT^{-1}(y))(T^{-1}(x)))\\
&=[TT^{-1}(x), TT^{-1}(y)]=[x, y].
\end{align*}
Conversely, let $(A,\diamond,\ast,\alpha)$ be a Hom-pre-$F$-manifold algebra and $(A,\cdot ,[\cdot,\cdot],\alpha)$ the sub-adjacent Hom-$F$-manifold algebra. Then the identity map $\Id$ is an $\huaO$-operator on $A$ with respect to the representation $(A;\ad,  L,\alpha)$.
\end{proof}
Let $(A,\mu,\alpha)$ be a Hom-algebra and   $\omega\in\wedge^2A^*$. Recall that $\omega$ is a symplectic structure on $A$ if it satisfies
\begin{align*}
    &\omega(\alpha(x),\alpha(y)) =\omega(x,y),\quad \quad  \displaystyle\circlearrowleft_{x,y,z}\omega(\mu(x,  y),\alpha(z)) =0.
\end{align*}

\begin{cor}
Let $(A,\cdot ,[\cdot,\cdot],\alpha )$ be a  coherence Hom-$F$-manifold algebra and  $\omega$ be a symplectic structure on $(A,\cdot,\a)$ and $(A,[\cdot,\cdot],\a)$.
Then there is a compatible Hom-pre-$F$-manifold algebra structure on $A$ given by
$$
\omega(x\diamond y,\alpha(z))=\omega(\alpha(y),x\cdot  z),\quad \omega(x\ast y,\alpha(z))=\omega(\alpha(y),[z,x] ).
$$
\end{cor}

 \begin{proof}
   Since $(A,\cdot ,[\cdot,\cdot],\alpha )$ is a  coherence Hom-$F$-manifold algebra, $(A^*;\ad^\star,-\huaL^\star,(\alpha^{-1})^*)$ is a representation of $A$. By the fact that $\omega$ is a symplectic structure, $(\omega^\sharp)^{-1}$ is an $\huaO$-operator on the commutative Hom-associative algebra $(A,\cdot,\alpha )$ with respect to the representation $(A^*;-\huaL^\star,(\alpha^{-1})^*)$, where $\omega^\sharp:A\longrightarrow A^*$ is defined by $\langle\omega^\sharp(x),y\rangle=\omega(x,y)$. Furthermore,  $(\omega^\sharp)^{-1}$ is an $\huaO$-operator on the Hom-Lie algebra $(A,[\cdot,\cdot],\alpha )$ with respect to the representation $(A^*;\ad^\star,(\alpha^{-1})^*)$. Thus, $(\omega^\sharp)^{-1}$ is an $\huaO$-operator on the  coherence Hom-$F$-manifold algebra $(A,\cdot ,[\cdot,\cdot] ,\alpha)$ with respect to the representation $(A^*;\ad^\star,-\huaL^\star,(\alpha^{-1})^*)$. By Proposition \ref{pro:nsc}, there is a compatible Hom-pre-$F$-manifold algebra structure on $A$ given as above.
 \end{proof}

\section{Hom-pre-Lie  formal deformation  of commutative Hom-pre-Lie algebras}\label{Hom-Pre-Lie  formal deformation  of commutative Hom-pre-Lie algebras}
In this section, we introduce the notion of Hom-pre-Lie formal deformations
of commutative Hom-associative algebras   and show that Hom-$F$-manifold algebras
are the corresponding semi-classical limits. Furthermore, we show that Hom-pre-Lie infinitesimal deformations and extensions of
 Hom-pre-Lie $n$-deformations to Hom-pre-Lie $(n+1)$-deformations of a commutative Hom-associative algebra $A$ are classified by the second and the  third cohomology groups of the Hom-pre-Lie algebra $A$.

Let $(A,\cdot,\a )$ be a commutative Hom-associative algebra. Recall that an { associative formal deformation}  of $A$
is a sequence of bilinear maps $\mu_n:A\times A\rightarrow A$ for $n\geqslant 0$ with $\mu_0$ being the commutative associative
algebra product $\cdot $ on $A$, such that the $\K[[t]]$-bilinear product $\cdot_t$ on $A[[t]]$ determined by
$$x\cdot_t y=\sum_{n=0}^\infty t^n\mu_n(x,y),\quad\forall~x,y\in A$$
is associative, where $A[[t]]$ is the set of formal power series of $t$ with coefficients in $A$. Given by
$$\{x,y\}=\mu_1(x,y)-\mu_1(y,x),\quad\forall~x,y\in A.$$
It is well known \cite{Makhlouf&Silvestrov} that $(A,\cdot ,\{\cdot,\cdot\},\a)$ is a Hom-Poisson algebra, called the { semi-classical limit} of $(A[[t]],\cdot_t,\a)$.

Since a Hom-associative algebra can be regarded as a Hom-pre-Lie algebra,
one may look for formal deformations of a commutative Hom-associative
algebra into Hom-pre-Lie algebras, that is, in the Hom-associative formal deformation, we replace the Hom-associative product
$\cdot_{t}$ by the Hom-pre-Lie product, and wonder what
additional structure will appear on $A$.   Surprisingly, we obtain
 the structure of the Hom-$F$-manifold algebra. Now we
give the definition of a Hom-pre-Lie formal deformation of a
commutative Hom-associative algebra.
\begin{defi}
  Let $(A,\cdot,\a)$ be a commutative Hom-associative algebra. A { Hom-pre-Lie formal deformation} of $A$
  is a sequence of bilinear maps $\mu_n:A\times A\rightarrow A$ for $n\geqslant 0$ with $\mu_0$ being the commutative Hom-associative algebra product
   $\cdot $ on $A$, such that the $\K[[t]]$-bilinear product $\cdot_t$ on $A[[t]]$  given by
$$x\cdot_t y=\sum_{n=0}^\infty t^n\mu_n(x,y),\quad\forall~x,y\in A$$
is a Hom-pre-Lie algebra structure.
\end{defi}
Note that the rule of Hom-pre-Lie algebra product $\cdot_t$ on $A[[t]]$ is equivalent to
  {\small\begin{equation}\label{eq:pre-Lie rule}
  \sum_{i+j=k}\big(\mu_i(\mu_j(x,y),\a(z))-\mu_i(\a(x),\mu_j(y,z))\big)=\sum_{i+j=k}\big(\mu_i(\mu_j(y,x),\a(z))-\mu_i(\a(y),\mu_j(x,z))\big),\quad\forall~k\geqslant 0.
  \end{equation}}

\begin{thm}\label{thm:deformation quantization}
Let $(A,\cdot,\a)$ be a commutative Hom-associative algebra and $(A[[t]],\cdot_t,\a)$ a Hom-pre-Lie formal deformation of $A$. Define
$$[x,y] =\mu_1(x,y)-\mu_1(y,x),\quad\forall~x,y\in A.$$
Then $(A,\cdot ,[\cdot,\cdot],\a )$ is a Hom-$F$-manifold algebra, called the { semi-classical limit}
of $(A[[t]],\cdot_t,\a)$. The Hom-pre-Lie algebra $(A[[t]],\cdot_t,\a)$ is called a { Hom-pre-Lie deformation quantization} of $(A,\cdot,\a )$.
\end{thm}

\begin{proof}
  Define the bracket $[\cdot,\cdot]_t$ on $A[[t]]$ by
$$[x,y]_t=x\cdot_t y-y\cdot_t x=t[x,y] +t^2(\mu_2(x,y)-\mu_2(y,x))+\cdots, \quad \forall~x,y\in A.$$
By the fact that $(A[[t]],\cdot_t,\a)$ is a Hom-pre-Lie algebra, $(A[[t]],[\cdot,\cdot]_t,\a)$ is a Hom-Lie algebra.

The $t^2$-terms of the Hom-Jacobi identity for $[\cdot,\cdot]_t$ gives the Hom-Jacobi identity for $[\cdot,\cdot]$.
Indeed,
\begin{align*}
\displaystyle\circlearrowleft_{x,y,z}[\a(x),[y,z]_t]_t&=\displaystyle\circlearrowleft_{x,y,z}[\a(x),t[y,z] +t^2(\mu_2(y,z)-\mu_2(z,y))+\cdots]_t,~~~~\forall t\\
&=t^2\displaystyle\circlearrowleft_{x,y,z} [\a(x),[y,z]]+\cdots,~~~~~~~~\forall t\\
&=0.
\end{align*}
 Thus $(A,[\cdot,\cdot],\a)$ is a Hom-Lie algebra.

For $k=1$ in \eqref{eq:pre-Lie rule}, by the commutativity of $\mu_0$, we have
\begin{align*}
 & \mu_0(\mu_1(x,y),\a(z))-\mu_0(\a(x),\mu_1(y,z))-\mu_1(\a(x),\mu_0(y,z))\\
  &=\mu_0(\mu_1(y,x),\a(z))-\mu_0(\a(y),\mu_1(x,z))-\mu_1(\a(y),\mu_0(x,z)).
\end{align*}
This is just the equality \eqref{eq:pseudo-pre-HM1} with $x\cdot  y=\mu_0(x,y)$ and $x\ast  y=\mu_1(x,y)$ for $x,y\in A$.
Then $(A,\cdot ,\ast,\a )$ is a Hom-$F$-manifold-admissible algebra. Thanks to  Theorem \ref{pro:pseudo-pre-HMA-HMA}, $(A,\cdot ,[\cdot,\cdot],\a)$ is a Hom-$F$-manifold algebra.
\end{proof}

In the following, we study Hom-pre-Lie $n$-deformations and pre-Lie infinitesimal deformations of   commutative Hom-associative algebras.
\begin{defi}
  Let $(A,\cdot,\a )$ be a commutative Hom-associative algebra. A { Hom-pre-Lie  $n$-deformation} of $A$ is a sequence of bilinear maps
   $\mu_k:A\times A\rightarrow A$ for $0\leq k\leq n$ with $\mu_0$ being the commutative Hom-associative algebra product $\cdot $ on $A$,
   such that the $\K[[t]]/(t^{n+1})$-bilinear product $\cdot_t$ on $A[[t]]/(t^{n+1})$ determined by
$$x\cdot_t y=\sum_{k=0}^nt^k\mu_k(x,y),\quad\forall~x,y\in A$$
is a Hom-pre-Lie algebra structure.
\end{defi}

We call a Hom-pre-Lie $1$-deformation of a commutative Hom-associative algebra $(A,\cdot,\a )$ a { Hom-pre-Lie infinitesimal deformation}
and denote it by $(A,\mu_1,\a)$.

By direct calculations, $(A,\mu_1,\a)$ is a Hom-pre-Lie infinitesimal deformation of a commutative Hom-associative algebra $(A,\cdot,\a )$
  if and only if for all $x,y,z\in A$
 \begin{eqnarray}
 \label{2-closed} &&\mu_1(x,y)\cdot  \a(z)-\a(x)\cdot  \mu_1(y,z)-\mu_1(x,y\cdot  \a(z))\nonumber\\
 &=&\mu_1(y,x)\cdot  \a(z)-\a(y)\cdot  \mu_1(x,z)-\mu_1(\a(y),x\cdot  z).
\end{eqnarray}
Equation $(\ref{2-closed})$ means that $\mu_1$ is a $2$-cocycle for the Hom-pre-Lie algebra $(A,\cdot )$, i.e. $\partial\mu_1=0$.

Two Hom-pre-Lie infinitesimal deformations  $A_t=(A,\mu_1,\a)$ and $A'_t=(A,\mu'_1,\a')$  of a
commutative Hom-associative algebra $(A,\cdot,\a)$ are said to be {
equivalent} if there exists a family of Hom-pre-Lie algebra
homomorphisms ${\Id}+t\varphi:A_t\longrightarrow A'_t$ modulo $t^2$. A Hom-pre-Lie infinitesimal deformation
is said to be { trivial} if there exists a family of Hom-pre-Lie
algebra homomorphisms ${\Id}+t\varphi:A_t\longrightarrow (A,\cdot,\a )$ modulo $t^2$.

By direct calculations, $A_t$ and $A'_t$  are
equivalent Hom-pre-Lie infinitesimal deformations if and only if
\begin{eqnarray}
\mu_1(x,y)-\mu_1'(x,y)&=&x\cdot  \varphi(y)+\varphi(x)\cdot  y-\varphi(x\cdot  y).\label{2-exact}
\end{eqnarray}
Equation $(\ref{2-exact})$ means that $\mu_1-\mu_1'=\partial \varphi$. Thus we have

\begin{thm}
 There is a one-to-one correspondence between the space of equivalence classes of Hom-pre-Lie infinitesimal deformations of $A$ and
  the second cohomology group $H^2(A,A)$.
\end{thm}

It is routine to check that

\begin{pro}
  Let $(A,\cdot,\a)$ be a commutative Hom-associative algebra such that $H^2(A,A)=0$. Then all Hom-pre-Lie infinitesimal deformations of $A$ are trivial.
\end{pro}

\begin{defi}
Let $\{\mu_1, \cdots,\mu_{n}\}$ be a Hom-pre-Lie $n$-deformation of a commutative Hom-associative algebra $(A,\cdot,\a )$. A Hom-pre-Lie $(n+1)$-deformation
of a commutative Hom-associative algebra $(A,\cdot,\a)$ given by $\{\mu_1,\cdots,\mu_n,\mu_{n+1}\}$ is called an { extension} of the Hom-pre-Lie $n$-deformation
 given by $\{\mu_1, \cdots,\mu_{n}\}$.
\end{defi}

\begin{thm}
  For any Hom-pre-Lie $n$-deformation of a commutative Hom-associative algebra $(A,\cdot,\a)$, the $\Theta_n\in\Hom(\otimes^3\g,  A)$ defined by
{\small\begin{equation}\label{eq:3-cocycle}
\Theta_n(x,y,z)=\sum_{i+j=n+1,i,j\geq 1}\big(\mu_i(\mu_j(x,y),\a(z))-\mu_i(\a(x),\mu_j(y,z))-\mu_i(\mu_j(y,x),\a(z))+\mu_i(\a(y),\mu_j(x,z))\big)
\end{equation}}
is a cocycle, i.e. $\partial \Theta_n=0$.

 Moreover, the Hom-pre-Lie $n$-deformation $\{\mu_1, \cdots,\mu_{n}\}$ extends into some Hom-pre-Lie $(n + 1)$-deformation
 if and only if $[\Theta_n]=0$ in $H_{\rm  }^3(A,A)$.
\end{thm}
\begin{proof}
It is obvious that
\begin{equation*}
\Theta_n(x,y,z)=-\Theta_n(y,x,z),\quad \forall~x,y,z\in A.
\end{equation*}
Thus $\Theta_n$ is an element of $C^3(A,A)$. It is straightforward to check that the cochain $\Theta_n\in C^3(A,A)$ is closed.

Assume that the Hom-pre-Lie $(n+1)$-deformation of a commutative Hom-associative algebra $(A,\cdot )$
 given by $\{\mu_1,\cdots,\mu_n,\mu_{n+1}\}$ is an extension of the Hom-pre-Lie $n$-deformation   given by $\{\mu_1,\cdots,\mu_n\}$, then we have
\begin{align*}
\begin{split}
  \label{eq:n+1 term}&\a(x)\cdot  \mu_{n+1}(y,z)-\a(y)\cdot  \mu_{n+1}(x,z)+ \mu_{n+1}(y,x)\cdot  \a(z)-\mu_{n+1}(x,y)\cdot  \a(z)+\mu_{n+1}(y,x)\cdot \a(z)  \\
   \nonumber&\quad-\mu_{n+1}(x,y)\cdot \a(z)=\sum_{i+j=n+1,i,j\geq 1}\big(\mu_i(\mu_j(x,y),\a(z))-\mu_i(\a(x),\mu_j(y,z))-\mu_i(\mu_j(y,x),\a(z)) \\
  \nonumber &+\mu_i(\a(y),\mu_j(x,z))\big).
  \end{split}
\end{align*}
It is obvious that the right-hand side of the above equality is just $\Theta_n(x,y,z)$. We can rewrite the above equality as
$$\partial \mu_{n+1}(x,y,z)=\Theta_n(x,y,z).$$
We conclude that, if a Hom-pre-Lie $n$-deformation of a commutative Hom-associative algebra $(A,\cdot,\a)$ extends to a Hom-pre-Lie $(n + 1)$-deformation,
then $\Theta_n$ is coboundary.

Conversely, if $\Theta_n$ is coboundary, then there exists an element $\psi\in C^2(A,A)$ such that
$$\partial \psi(x,y,z)=\Theta_n(x,y,z).$$
It is not hard to check that $\{\mu_1,\cdots,\mu_n,\mu_{n+1}\}$ with $\mu_{n+1}=\psi$ generates a
 Hom-pre-Lie $(n+1)$-deformation of $(A,\cdot,\a )$ and thus this Hom-pre-Lie $(n+1)$-deformation is an extension of the Hom-pre-Lie $n$-deformation
 given by $\{\mu_1, \cdots,\mu_{n}\}$.
\end{proof}


\begin{thebibliography}{999}
\bibitem{Aizawa&Sato} N. Aizawa, H. Sato, q-deformation of the Virasoro algebra with central extension, Phys.
Lett. B 256 (1991), 185–190.

\bibitem{Aguiar}  M.Aguiar,  Pre-poisson algebras. Lett. in Math. Physi., 54(4), 263-277(2000).

\bibitem{Ammar&Mabrouk&Makhlouf} F. Ammar, S. Mabrouk, A. Makhlouf, Representations and cohomology of $n$-ary multiplicative Hom-Nambu–Lie algebras, J.Geom. Phys. 61 (10) (2011) 1898–1913.

\bibitem{Bai&Guo&Ni} C. Bai, L. Guo, X. Ni, Generalizations of the classical Yang–Baxter equation and O-operators, J. Math. Phys. 52 (2011) 063515.

\bibitem{Bai&Guo&Ni2010} C. Bai, L. Guo, X. Ni, Nonabelian generalized Lax pairs, the classical Yang–Baxter equation and Post-Lie algebras, Comm. Math. Phys. 297 (2010) 553–596.

\bibitem{Baxter} G. Baxter, An analytic problem whose solution follows from a simple algebraic identity, Pacific J. Math. 10 (1960) 731–742.

\bibitem{Benayadi&Makhlouf} S.Benayadi, A. Makhlouf, Hom-Lie algebras with symmetric invariant nondegenerate bilinear forms, J. of Geom. and Phy. 76 (2014): 38-60.

\bibitem{Cartier} P. Cartier, On the structure of free Baxter algebras, Adv. Math. 9 (1972) 253–265.

\bibitem{Connes&Kreimer} A. Connes, D. Kreimer, Hopf algebras, renormalisation and noncommutative geometry, Comm. Math. Phys. 199 (1988) 203-242.

\bibitem{Cai&Sheng} L. Cai, Y. Sheng, Purely Hom-Lie bialgebras. Science China Mathematics, 61(9), 1553-1566 (2018).

\bibitem{DS04}
L. David and I. A. B. Strachan, Compatible metrics on a manifold and non local bi-Hamiltonian structures. \emph{Int. Math. Res. Not.} 66 (2004), 3533-3557.

\bibitem{DS11}
L. David, and I. A. B. Strachan, Dubrovins duality for $F$-manifolds with eventual identities. \emph{Adv. Math.} 226 (2011), 4031-4060.

\bibitem{Dot} V. Dotsenko, Algebraic structures of F-manifolds via pre-Lie algebras. \emph{Ann. Mat. Pura Appl.} 198 (2019), 517-527.

\bibitem{Fard&Guo&Kreimer} K. Ebrahimi-Fard, L. Guo, D. Kreimer, Spitzer’s identity and the algebraic Birkhoff decomposition in pQFT, J. Phys. A 37 (2004) 11037-11052.

\bibitem{Fard&Guo&Manchon} K. Ebrahimi-Fard, L. Guo, D. Manchon, Birkhoff type decompositions and the Baker-Campbell-Hausdorff recursion, Comm. Math. Phys. 267 (2006)821–845.


\bibitem{Guo2009} L. Guo, What is a Rota-Baxter algebra, Notices Amer. Math. Soc. 56 (2009) 1436-1437.

\bibitem{Guo2012} L. Guo, Introduction to Rota-Baxter Algebra, International Press and Higher Education Press, 2012.

\bibitem{Guo&Keigher} L. Guo, W. Keigher, Baxter algebras and shuffle products, Adv. Math. 150 (2000)117-149.

\bibitem{Guo&Zhang} L. Guo, B. Zhang, Renormalization of multiple zeta values, J. Algebra 319 (2008) 3770-3809.

\bibitem{Guo&Zhang&Wang}  S. Guo, X. Zhang, S. Wang, Manin triples and quasitriangular structures of Hom-Poisson bialgebras. arXiv preprint arXiv:1807.06412  (2018).

 \bibitem{Her02}
C. Hertling, Frobenius Manifolds and Moduli Spaces for Singularities. \emph{ Cambridge Tracts in Math.} Cambridge
University Press, 2002.

\bibitem{HerMa}
 C. Hertling and Y. I. Manin, Weak Frobenius manifolds. \emph{Int. Math. Res. Not.} 6 (1999), 277-286.

\bibitem{Hartwig&Larsson&Silvestrov} J. T. Hartwig, D. Larsson, S. D. Silvestrov, Deformations of Lie algebras using $\sigma$-
derivations, J. Algebra 295 (2006), 314-361.

\bibitem{Hu} N. Hu, $q$-Witt algebras, $q$-Lie algebras, $q$-holomorph structure and representations, Algebra
Colloq. 6 (1999), 51–70.

\bibitem{K}
B. A. Kupershmidt, What a classical $r$-matrix really is, \emph {J.
Nonlinear Math. Phys.} 6 (1999), 448-488.

\bibitem{Larsson&Silvestrov} D. Larsson, S. D. Silvestrov, Quasi-hom-Lie algebras, central extensions and $2$-cocycle-like
identities, J. Algebra 288 (2005), 321–344.

\bibitem{LYP}
Y. P. Lee, Quantum K-theory, I: Foundations. \emph{Duke Math. J.} 121 (3) (2004), 389-424.

\bibitem{Liu&Song&Tang}  S. Liu,  L. Song, R. Tang, Representations and cohomologies of regular Hom-pre-Lie algebras. Journal of Algebra and Its Applications, 19(08), 2050149  (2020).

\bibitem{Liu&Sheng&Bai} J. Liu, Y. Sheng and C. Bai, F-manifold algebras and deformation quantization via pre-Lie algebras. J.
Algebra 559 (2020), 467-495.

 \bibitem{LPR11}
P. Lorenzoni, M. Pedroni and A. Raimondo, $F$-manifolds and integrable systems of hydrodynamic type. \emph{Arch. Math.} (Brno) 47 (2011), 163-180.

\bibitem{Manchon&Paycha} D. Manchon, S. Paycha, Nested sums of symbols and renormalised multiple zeta values, Int. Math. Res. Not. (2010) 4628C4697.

\bibitem{Makhlouf&Silvestrov} A. Makhlouf, S. D. Silvestrov, Hom-algebras structures, J. Gen. Lie Theory Appl. 2 (2008),
51–64.

\bibitem{Makhlouf&Silvestrov1} A.Makhlouf and S. Silvestrov, Notes on $1$-parameter formal deformations of Hom-associative and Hom-Lie algebras, Forum Mathematicum. Vol. 22. No. 4. De Gruyter, 2010.

\bibitem{ms2}
A. Makhlouf and S. Silvestrov, Hom-algebras and Hom-coalgebras, J. Alg. Appl. 9 (2010) 1-37.

\bibitem{Merku}
S. A. Merkulov, Operads, deformation theory and $F$-manifolds.  {Asp. Math.} 36 (2004), 213-251.

\bibitem{Ming&Chen&Li} D. Ming, Z Chen,  J. Li, $F$-manifold color algebras,  arXiv:2101.00959 (2020).

\bibitem{Rota1969} G.C. Rota, Baxter algebras and combinatorial identities $I$, $II$, Bull. Amer. Math. Soc. 75 (1969) 325-334.

\bibitem{Rota1995} G.C. Rota, Baxter operators, an introduction. Gian-Carlo Rota on combinatorics, in: Contemp. Mathematicians, Birkhäuser Boston, Boston, MA, 1995, pp. 504-512.

\bibitem{Sheng} Y. Sheng, Representations of hom-Lie algebras, Algebr. Represent. Theory 15 (2012) 1081–1098.

\bibitem{Sheng&Chen} Y. Sheng, D. Chen, Hom-Lie 2-algebras, J. Algebra 376 (2013) 174–195.

\bibitem{Yau2009} D.Yau, Hom-algebras and homology, J. Lie Theory 19 (2009) 409–421.

\bibitem{Yau2010} D. Yau, Non-commutative hom-Poisson algebras,  arXiv:1010.3408 (2010).

\bibitem{Yau2012} D. Yau, Hom-Malcev, Hom-alternative and Hom-Jordan algebras. International J. of
Alg. Volume 11 (2012) 177-217.

\end{thebibliography}
\end{document}